\documentclass{article}
\usepackage{amsmath, amssymb, amsfonts, amsthm, url, graphicx}
\usepackage{xcolor}

\newtheorem{theorem}{Theorem}[section]
\newtheorem{lemma}[theorem]{Lemma}
\newtheorem{proposition}[theorem]{Proposition}
\newtheorem{corollary}[theorem]{Corollary}

\theoremstyle{definition}
\newtheorem{definition}[theorem]{Definition}
\newtheorem{remark}[theorem]{Remark}
\newtheorem{example}[theorem]{Example}

\begin{document}

\begin{center}
\huge{Bifurcations  in the family of billiards associated with the curvature flow}
\vspace{0.5cm}

\large{J.G. Damasceno,  M.J. Dias Carneiro and C. Salazar }\\

\end{center}
\vspace{0.5cm}

\textbf{{Abstract}}
We describe some dynamical properties of one parameter families of billiards on convex curves (ovals) which are deformed by the curvature (curve-shortening) flow. We obtain the bifurcations of the period two orbits and some special non-Birkhoff orbits, the normal periodic orbits. We prove the destruction of non-convex caustics of the ellipse by deforming it through the curvature flow.

\footnote{\textbf{ Classification}: 37A05;37E40;37J45

 }

\section{ Introduction}

This work aims to describe some dynamical properties of the family of billiard maps associated with a one-parameter family of curves satisfying the curvature flow or the curve shortening flow.

This flow was studied in a series of papers by M. Gage and R. Hamilton \cite{G1},\cite{GH}, and M. Grayson \cite{Grayson} that were published in the late 1980s. These works deal mainly with the long-time behavior of regular closed plane curves which deform in the direction of the curvature vectors. The curves generally shrink to a point, but they become increasingly ''round".  This geometrical feature is proved by considering a {\it normalized flow} such that the enclosed area is constant equals to $\pi$.  In this case, the curvature converges uniformly to one. Grayson proved that any simple closed curve becomes convex and then evolves as a family of convex curves approximating the circle of radius one.

Therefore it is natural to ask what type of changes occur in the billiard maps associated with this family of curves. In other words, changes in the dynamical properties can be observed as the curve evolves by the curvature flow.

Billard maps generally may have very complicated behavior, since it may contain chaotic regions, the  Birkhoff instability regions. On the other hand, billiard on circles are trivial in the sense that the phase space is foliated by invariant circles and the map restricted to each circle is a rotation. Hence as the curves
evolves the dynamics should tend to a less  (although not monotonically) complicated behavior.

The main result of this paper is analogous to  Theorem 1.1 of \cite{D-DC-RR} for homotopically trivial invariant curves around the elliptic period 2 orbit of the ellipse. The caustics corresponding to such invariant curves are  co-focal hyperboles:

\begin{theorem}
The normalized curvature flow breaks all resonant hyperbolic caustics of the billiard map on the ellipse.
\end{theorem}

 %In this work,  concentrate the analysis on the primary instability region, that is, the annulus in the phase space that contains the period two orbits (diameters) and  does not contain rotational invariant curves. For this, we start by studying the bifurcations of the period two orbits-hyperbolic, parabolic or elliptic-along the flow.

We also study the bifurcations of another type of periodic orbits, the normal periodic orbits $NP(2n)$ for a family of ovals satisfying $C_t$ satisfying the normalized curve shortening flow. 

 Similar to the period two orbits, these periodic orbits start perpendicular to the boundary curve and after some reflections, they hit the boundary again orthogonally. Thus, after the reflection they return tracing back the same polygonal arriving to the initial point.

 There is a geometric condition for the existence of these orbits in terms of curves similar to evolutes.  We show in Theorem 4.10, that, as the boundary curve evolves by the curvature flow,  these orbits disappear gradually.

 In other words, in Theorem we prove:

{\it If $X(t,s)$ is a family of convex curves satisfying the curvature flow, then for any integer $N>1$, there is a real number $T$ such that, for every  $t>T$, the billiard map associated with the curve does not have a normal periodic orbit of period $2n \leq  2N$.}

This also may be interpreted as {\it the long term breaking of a family of invariant rotational resonant curve around elliptic diameters}.

The paper is organized as follows: in Section 2 we describe the main properties of the curvature flow and briefly introduce the billiard map.

In Section 3, we discuss the example of the ellipse. In Section 4, the bifurcations of period two orbits (diameters) and the normal periodic orbits are considered. The last section is dedicated to the proof of Theorem 1.1.

\section{ The curve shortening flow and billiards}
\subsection{The curve shortening flow}
The main references for this section are  \cite{G1},\cite{GH} and \cite{chineses} .

We consider family of regular ($C^r$, $r \geq 2)$, simple, closed curves in ${\mathbb R}^2$, with  Euclidian metric, $X(t,u)$ that satisfies the curve shortening flow:

$$\frac{\partial X} {\partial t} (t,u)= k(t,u) N(t,u)$$

Where $k(t,u)$ is the curvature and $N(t,u)$ is the normal (unitary) vector pointing inward. Observe that $u$ is not the arc-length parameter.

The main results of this flow are the following:

\begin{theorem} [Gage-Hamilton] The curve shortening flow preserves convexity and shrinks any closed simple convex curve to a point.

\end{theorem}

\begin{theorem}
[Grayson] Starting with any closed curve it becomes convex before it shrinks to a point.
\end{theorem}

Let $A(t)=  \frac{1}{2} \int_{\gamma(t)} xdy-ydx$  be the area enclosed by the curve $\gamma(t)$.Then $A'(t)=-2\pi$.

This result tells two things:
\begin{itemize}
    \item[a)] The existence time $\tau$ of the solution  is given by $\tau = \frac{A(0)}{2 \pi};$
    \item[b)]  there is a natural normalization of the flow by taking a homothety $Y(t,u)= \frac{1}{\sqrt{\tau-2t}}X(t,u)$ such that the area of the region enclosed by the curve $Y(t,.)$ is constant equal to $\pi$.
\end{itemize}

There is also a time normalization given by $\bar t= \frac{1}{2} ln[\frac{\tau}{\tau-t}]$, so that $t \rightarrow \tau$ implies
${\bar t} \rightarrow \infty$.

For the normalized flow, the area enclosed by the curves is constant equal to $\pi$, and Gage and Hamilton prove that the family converges uniformly to a circle of radius 1.

In this work we consider {\it Normalized flow} and observe that the properties of the billiards maps do not change by the normalization process.

We state some facts  which are going to be useful in later.

Convex curves may be parametrized by the angle $\theta$ which its tangent makes with a fixed direction, denoting by $T(\theta)$ the unitary tangent vector and by $N(\theta)$ its normal vector pointing inwards, we get the Frenet's formula:

$$T'(\theta)=N(\theta)  \ \ \ \   N'(\theta)=-T(\theta)$$

Throughout this paper $g'$ will always denote the derivative of a function $g(t,\theta)$ with respect to the variable $\theta$ and $g_{t}$ will denote the derivative with respect to $t$. 
Assuming that the domain $U$ bounded by $X$ contains the origin $O=(0,0)$ we let  $h(\theta)= -\langle X(\theta),N(\theta)\rangle$ be the support function with respect to $O$. Then $X(\theta)=h'(\theta) T(\theta)-h(\theta)N(\theta)$ 
and $h''(\theta)=R(\theta)-h(\theta)$, for $R(\theta)=\frac{1}{k(\theta)}$, the radius of curvature. 

Using this representation, the equation for the normalized flow, with $t$ the normalized time,  is

$$X_{t}(t, \theta)=[h'(t,\theta)-k'(t,\theta)] T(\theta)+[k(t,\theta)-h(t,\theta)]N(\theta)$$

\noindent As for the normalized geometric quantities we have the following evolution equations:

$$k_{t}(t,\theta)={[k(t,\theta)]}^{2} k''(t,\theta) +{[k(t,\theta)]}^{3}- k(t,\theta)$$

$$h_{t}(t,\theta)=h(t,\theta)- k(t,\theta)$$

One of the results which have been proven by Gage and Hamilton  is the uniform  exponential decay of the derivatives of the curvature:
\begin{theorem}
	For every $n\geq 1$ and $0<\alpha<1$, there exists a constant $C(n)$ such that:
	$$    \Vert k^{(n)} \Vert _{\infty} \leq C(n) e^{-2\alpha t}$$
\end{theorem}

Moreover, for any $\epsilon>0$, there is $C(\epsilon)>0$  and  $t_{\epsilon}>0$ such that  $$\vert  k(\theta,t) -1 \vert < C(\epsilon)e^{-(2-\epsilon)t}, \forall t > t_{\epsilon}$$

Analogously for the radius of curvature $R(u)$.

  \subsection{Strictly convex billiards}

%Let $U$ be a region of $\mathbb{R}^2$ whose boundary is a simple closed curve $X$ which is assumed to be  convex and $C^k$, $k\geq 3$.

The billiard map is obtained from the free motion of a point particle on the planar region $U$  being reflected
elastically at the impacts on its boundary $X$. The trajectories are polygonals on this planar
region.
Since the motion is free inside the region, it is completely determined by the points of impact at $X$ and the direction of the motion. This direction is measured by the angle between the line and the tangent of the curve.

Therefore, to each oval parametrized by the angle $\theta$ of a tangent vector to a fixed direction, we
associated a billiard map $B$  from the cylinder ${\mathbb S}^1\times (0,\pi)$ into itself,
defined  in the  following way:  given a point $ x \in X$, let $\rho_{(x,\phi)}$ be the ray through $x$ making a angle $\phi$ with the vector tangent at $x$.
The billiard map assigns to $(x,\phi)$ the point $(x_1,\phi_1)$, where $x_1$ is the intersection of $\rho_{(x,\phi)}$ with the curve $X$, and $\phi_1$ is the reflected angle between $\rho_{(x,\phi)}$ and the tangent to $X$ at $x_1$.

If the  curve $X$ is parametrized by the arc-length $s$, then the billiard map is a monotonous twist map and has a generating function:
$L(s,s')=||X(s')-X(s) ||$ where $||\cdot||^2= \langle \cdot,\cdot\rangle$, the euclidean inner product. This means:
$$-\frac{\partial L}{\partial s}=\cos(\phi)$$

$$\frac{\partial L}{\partial s'}=\cos(\phi_{1})$$

%Birkhoff  called the region bounded by two invariant rotational (i. e homotopically non trivial) curves, with no other invariant
%rotational curves inside, an instability region.

\section{Bifurcations of normal periodic points along the curvature flow}
\subsection{Period two orbits  }

Let $X(t, \theta)$ be a family of simple convex plane curves parametrized by $\theta$, the angle between its tangent vector and a fixed direction. Suppose that $X$ satisfies the normalized curve shortening equation.

%$$X_{t}(t,\theta)= [h'(t, \theta)-k'(t, \theta)]T(\theta) +[k(t, \theta)-h(t, \theta)]N(\theta)$$

%We recall that $h=-\langle X,N\rangle$ is the support function , so $X(t,\theta)=h'(t, \theta) T(\theta) -h(t, \theta) N(\theta)$ with $h''(t, \theta) +h(t, \theta)=R(t, \theta)$, $R=\frac{1}{k}$, the radius of curvature.

%The partial differential equation for the curvature function of the  family of curves satisfying the curvature flow is  the following:

%$$\displaystyle \frac{\partial}{\partial t}k(t,\theta)=k^2(t,\theta)\frac{\partial^2}{\partial \theta^2}k(t,\theta)+k^3(t,\theta)-k(t,\theta)$$

%The variational principle that defines the billiard is the chord $L(t,\theta_{1},\theta_{2})= \Vert X(t,\theta_{2})-X(t,\theta_{1}) \Vert$.

It is well known that the period two orbits are the critical points of $\ell(t,\theta)=\Vert X(t,\theta + \pi)-X(t,\theta) \Vert$.

In other words, the value $\theta$ corresponds to a periodic orbit of period two if and only if  $f(t,\theta)=\langle T(\theta), X(t,\theta + \pi)-X(t,\theta) \rangle=-h'(t, \theta+\pi)-h'(t, \theta)=0$ 

The classification of period two orbits is also well-known:

\begin{itemize}
\item Hyperbolic: $ \ell(t, \theta)-[R(t, \theta)+R(t, \theta+\pi)] > 0$ or $ \ell(t, \theta)-[R(t, \theta)+R(t, \theta+\pi)] <0$ and $[ \ell(t, \theta)-R(t, \theta)][\ell(t,\theta)-R(t, \theta+\pi)] >0$

 \item Elliptic: $ \ell(t, \theta)-[R(t, \theta)+R(t, \theta+\pi)] < 0$ and  $[ \ell(t, \theta)-R(t, \theta)][\ell(t,\theta)-R(t, \theta+\pi)] >0$

 \item Parabolic: $ \ell(t, \theta)-[R(t, \theta)+R(t, \theta+\pi)]= 0$ or  $ \ell(t, \theta)-[R(t, \theta)+R(t, \theta+\pi)] <0$ and $[ \ell(t, \theta)-R(t, \theta)][\ell(t,\theta)-R(t, \theta+\pi)] =0$

\end{itemize}

\begin{remark}  For sufficiently large $t$ and  for every $\theta$,  $ \ell(t, \theta)-R(t, \theta) >0$.
\end{remark}
Indeed,
let $w(t)= \int_{{\mathbb S}^{1}} log(h(t,\theta)) d\theta$.

Then $$w'(t)=\int_{{\mathbb S}^{1}} \frac{(h(t,\theta)-k(t,\theta)}{h(t,\theta)} d\theta=2\pi -\int_{{\mathbb S}^{1}} \frac{k(t,\theta)}{h(t,\theta)} d\theta$$

Substituting $$\frac{k(t,\theta)}{h(t,\theta)} =\frac{[{h(t,\theta)-k(t,\theta)]}^2}{h(t,\theta)k(t,\theta)} +2-\frac{h(t,\theta)}{k(t,\theta)}$$

and using the fact $\int_{{\mathbb S}^{1}}\frac{h(t,\theta)}{k(t,\theta)} d\theta=2 Area(U)= 2\pi$ to obtain

$$w'(t)=-\int_{{\mathbb S}^{1}}\frac{{[h(t,\theta)-k(t,\theta)]}^2}{h(t,\theta)k(t,\theta)} d\theta <0$$

So $w(t)$ is strictly decreasing and bounded from below. Hence $w'(t)$ converges to $0$. Since $k(t,\theta)$ converges uniformly to 1 we obtain that $h(t,\theta)$ also converges uniformly to 1.

As a consequence we obtain $\ell(t, \theta)-R(t, \theta)=h(t,\theta +\pi)+h(t,\theta)- R(t,\theta)$ converges to 1.

Hence, by assuming $t>0$ large enough, the above conditions may be formulated in a simplified form:

\begin{itemize}
\item Hyperbolic: $ \ell(t, \theta)-[R(t, \theta)+R(t, \theta+\pi)] > 0$

 \item Elliptic: $ \ell(t, \theta)-[R(t, \theta)+R(t, \theta+\pi)] < 0$

 \item Parabolic: $ \ell(t, \theta)-[R(t, \theta)+R(t, \theta+\pi)]= 0$

\end{itemize}

Therefore, the dynamical properties of the period two orbits, and their bifurcations are obtained by analyzing the curve $f(t,\theta)=0$. 

Let us assume that $\theta_{0}$ is a period two orbit for the curve $X(t_{0}, \theta)$, that is:  
$ f(t_{0},\theta_{0})=\langle X(t_{0},x_{0})-X(t_{0},x_{0}),T(\theta) \rangle=0=-[h'(t_{0},\theta_{0}+\pi)+h'(t_{0},\theta_{0})]$. 
Observe that if $ f'(t_{0},\theta_{0})= \ell(t, \theta)-[R(t, \theta)+R(t, \theta+\pi)] \neq 0$ then by the Implicit function theorem there is a $C^{r}$ family of nondegenerate period-two orbits $(t,\theta(t))$, Hyperbolic or Elliptic according to the sign of $ f'(t_{0},\theta_{0})$.

 The {\it  parabolic} period-two orbits  are   defined  by:

$ f(t_{0},\theta_{0})=0$  and $f'(t,\theta) =h(t, \theta)+h(t, \theta + \pi)- [R(t, \theta)+R(t, \theta+\pi)]$

Recall the equation of the evolute of the curve $X(t,.)$: 
$E(t,\theta)=X(t,\theta)+R(t,\theta)N(\theta)=h'(t,\theta)T(\theta)+[R(t,\theta)-h(t,\theta)]N(\theta)$.

The following difference characterizes the period-two orbits: 
  $$E(t,\theta+\pi)-E(t,\theta)=f(t,\theta)T(\theta)+f'(t,\theta)N(\theta)$$

%The term $f'(t,\theta)$ measures the distance between the two points of the evolute. 

 A parabolic period-two orbit is characterized by $E(t_{0},\theta_{0}+\pi)-E(t_{0},\theta_{0})=0$.  ($ f'(t,\theta)$  measures the speed of crossing of the two points).

%\textcolor{red}{acho que o primeiro $T(\theta)$ deveria ser $T(\theta+\pi)$ resposta $T(\theta+\pi)=-T(\theta)$ }
\begin{eqnarray}\label{xxx}
\begin{array}{ll}
f_{t}(t,\theta) %&= <T(\theta), [h_{\theta}(\theta+ \pi) -k'(\theta+ \pi)]T(\theta +\pi ) +[h(\theta+ \pi) -k(\theta +\pi)]N(\theta+\pi)> \\\\
%&+
%<T(\theta), [h_{\theta}(\theta) -k'(\theta)]T(\theta  ) +[h(\theta+ \pi) -k(\theta +\pi)]N(\theta)> \\\\

%&= - <T(\theta), [h_{\theta}(\theta+ \pi) -k'(\theta+ \pi)]T(\theta ) +
%<T(\theta), [h_{\theta}(\theta) -k'(\theta)]T(\theta  ) > \\\\

 %&= -  [h'(\theta+ \pi) -k'(\theta+ \pi)]- [h'(\theta) -k'(\theta)]\\\\

& = -  h'(t, \theta+ \pi) - h'(t, \theta) + k'(t, \theta+ \pi)+k'(t, \theta)\\\\
&=-f(t, \theta)+k'(t, \theta+ \pi)+k'(t, \theta)

\end{array}
\end{eqnarray}

%Therefore if $f'(t_{0},\theta_{0}) \neq 0$ the corresponding period two orbit is nondegenerate, hyperbolic or elliptic.
%By the Implicit Function Theorem, there is a one parameter family of nondegenerate period two orbits,  $t \mapsto \theta(t)$, for $t$ in a neighborhood of $t_{0}$.

%In particular, if $\theta_{0}$ is a local nondegenerate  maximum of the chord function $\ell(\theta)$, then $(\theta, \frac{\pi}{2})$  is an hyperbolic period two orbit and we obtain a one parameter family of hyperbolic period two points.

%As a side remark, it is easy to see that in this case, $f'(t,\theta)$ essentially measures the angle between unstable and stable directions (eigenvectors)  at $( \theta(t),\frac{\pi}{2})$ and from equation (\ref{xxx}) we find that  as $t \rightarrow \infty$ it converges to zero with speed:

%$$f'_{t}(t,\theta)=-f'(t,\theta)+k''(t, \theta+ \pi)+k''(t, \theta)$$

%If   $f'(t_{0},\theta_{0}) = 0$  but $k'(t_{0}, \theta_{0}+ \pi)+k'(t_{0}, \theta_{0}) \neq 0$, then there is $\epsilon>0$ such that  $ t_{0}< t < t_{0}+\epsilon$, there is no  period two orbit for the curve $X(t,\theta)$ in a neighborhood of $\theta_{0}$. A typical  bifurcation of a nondegenerate parabolic periodic point.

 Integrating we obtain:

$$f(t, \theta)= e^{t_0-t} f(t_{0}, \theta) +  e^{t_0-t} \{ \int_{t_{0}}^{t} e^{s} [ k'(s, \theta+ \pi)+k'(s, \theta)] ds\}$$

 Hence if  $\int_{t_{0}}^{t} e^{s} [ k'(s, \theta_{0}+ \pi)+k'(s, \theta_{0})] ds \neq 0$  it follows  that $f(t, \theta)\neq 0$ for $t>t_{0}$ and $\theta$ near $\theta_{0}$ and there is no  period two orbit for the curve $X(t,\theta)$.  A typical  bifurcation of a  parabolic periodic point.
 %the same above conclusion holds, namely, there is no period two orbit near $\theta_{0}$.

%Therefore, we are interested in those diameters such that $ k'(t, \theta_{0}+ \pi)+k'(s, \theta_{0})=0$ for every $t>t_{0}$.
%The half line $(t,\theta_{0})$, for $t\geq t_{0}$ is a curve of period-two orbits.

%$$f'(t_{0},\theta_{0}) =0$$

%$$k'(t_{0}, \theta_{0}+ \pi)+k'(t_{0}, \theta_{0}) =0, $$

%So that, actually we have $f(t,\theta_{0}) \equiv 0$ a family of parabolic period two orbits evolving along the family.

\begin{proposition}
Let  $\theta_{0}$ be a parabolic orbit of period two for the billiard associated with the curve $X(t_{0},.)$. Let us assume that there exists an  integer $n>1$ such that
$k^{(n)}(t_{0},\theta_{0})+ k^{(n)}(t_{0},\theta_{0} + \pi) \neq 0$ .

Then there is a neighborhood of $(t_{0}, \theta_{0})$ such that, any periodic orbit of period two in $V_{0} -\{ (t_{0},\theta_{0}) \}$  is non-degenerate, elliptic or hyperbolic.
\end{proposition}
\begin{proof}

If $n=2$, then  $k''(s,\theta_{0})+ k''(s,\theta_{0} + \pi) \neq 0$ for $s> t_{0}$ small. Hence

$$f'(t, \theta_{0})= e^{t_0-t} f'(t_{0}, \theta_{0}) +  e^{t_0-t} \{ \int_{t_{0}}^{t} e^{s} [ k''(s, \theta_{0}+ \pi)+k''(s, \theta_{0})] ds\} \neq 0$$

This implies that any period two orbit for $X(t,.)$ near $(t_{0}, \theta_{0})$ is hyperbolic or elliptic.

For general $n$, the argument is similar. Let $n > 2$ be the smallest integer such that $k^{(n)}(t_{0},\theta_{0})+ k^{(n)}(t_{0},\theta_{0} + \pi) \neq 0$

By the Taylor expansion we get  $$k''(t_{0},\theta_{})+ k''(t_{0},\theta + \pi) = [k^{(n)}(t_{0},\theta_{0})+ k^{(n)}(t_{0},\theta_{0} + \pi)] \frac{{(\theta -\theta_{0})}^{n}}{n!}+ O(n)$$

Hence we also get an isolated zero at $\theta_{0}$ and  there is a neighborhood $V_{0}$ of $(t_{0}, \theta_{0})$ such that  $k''(t,\theta)+ k''(t,\theta + \pi) \neq 0$ in
$V-(t_{0}.\theta_{0})$.

By applying the same reasoning as above we get $f'(t, \theta)\neq 0$ and the conclusion of the statement holds, any periodic orbit of period two in $V_{0} -(t_{0},\theta_{0})$ is non-degenerate, elliptic or hyperbolic.
\end{proof}

%{\bf Evolutes and $NP(2)$}:

%Let us first suppose that there exists $t_{1}\geq t_{0}$ such that $f(t_{1},\theta_{0})=0$ and $ f_{t}'(t_{1},\theta_{0})=k'((t_{1},\theta_{0}+\pi)+k'(t_{1},\theta_{0}) \neq 0$, suppose that  is the smallest such number.  Then, $f(t,\theta)=0$ defines locally a curve $(t(\theta),\theta)$ of period-two orbits which is tangent to the vertical  line $t=t_{1}$ at 
%$(t_{1},\theta_{0})$. 
%This implies that for $t>t(\theta)$ there is not a two-periodic orbit near $\theta_{0}$. A typical bifurcation of a parabolic point.

If $\forall t \geq t_{0}$ such that $f(t,\theta_{0})=0$, we have $k'((t,\theta_{0}+\pi)+k'(t,\theta_{0}) =0$, then the line 
$\theta=\theta_{0}, t\geq t_{0}$ defines a family of parabolic period-two orbits.

 In other words the point $E(t,\theta_{0}+\pi)=E(t,\theta_{0})$ moves in the chord from $X(t,\theta_{0})$ to $X(t,\theta_{0}+\pi)$.

Since $R(t,\theta) \rightarrow 1$ uniformly as  $t \rightarrow \infty$  the width $h(t,\theta_{0}+\pi)+h(t,\theta_{0}) \rightarrow 2$ and $ E(t,\theta_{0})$ converges to the origin, 
the middle-point of the chord.

%If $f(t_{0},\theta_{0})=0$  and $f'(t_{0},\theta_{0})\neq 0$, a local family of period-two points is defined by a function  $t \mapsto \theta(t)$ and its dynamical behavior is characterized by $t \mapsto f'(t,\theta(t))$.

%If for some $t_{1}>t_{0}$ we have $f'(t_{1},\theta_{1})= 0$, a parabolic period two orbit, then we apply the above analysis.

\begin{example} Curve of constant width
\end{example}
As an example of the above evolution of diameters, let us consider a curve of constant width. In other words, a closed simple convex curve such that any point of type $(\theta,\frac{\pi}{2})$ is a diameter. This means that in the phase space of the billiard map, the straight line $\Gamma \{ \phi=\frac{\pi}{2} \}$ is invariant.

 The condition of constant width is equivalent to:

$f(t,\theta)=\langle X(\theta+\pi)-X(\theta), T(\theta)\rangle =0$ or 
$ h(\theta)+h(\theta + \pi)=d, \forall \theta \in [0,2\pi]$, $d$ a positive number.

If   $f_{t}(0,\theta_{0})\neq 0$, for  some $\theta_{0}$, then for  $t>0$ small, the curve $X(t,\theta)$ is not of constant width anymore. Or, the invariant line is destroyed.

Hence, let us assume that the initial curve of constant width also satisfies $f_{t}(0,\theta_{0})= k'(0, \theta+ \pi)+k'(0, \theta)=0, \forall \theta$. 
%Equivalently  $k(0, \theta+ \pi)+k(0, \theta)=h(\theta+\pi)+h(\theta)=d$. 

We claim that such a curve is already the circle of radius one, that is the stationary curve. To see this,  we recall that $h''(\theta)=R(\theta)-h(\theta)$, hence, for curves of constant width we have $R(\theta+\pi)-h(\theta+\pi)+R(\theta)-h(\theta)=0$.
If this is the case,  $R(0,\theta+\pi)+R(0,\theta)=d$ and, using $k'(0, \theta+ \pi)+k'(0, \theta)=0$, we obtain that $k(0,\theta)$, a circle with area $\pi$.

%k(0, \theta+ \pi)+k(0, \theta)$. Since  by continuity, there is $\theta_0$ such that $R(0,\theta_0)=R(0,\theta_0+\pi)=1$ we obtain that the width of the intial curve is $d= 2$.

%Therefore, for such a curve, we have $$R(0,\theta+\pi)+R(0,\theta)=2=\frac{1}{R(0,\theta+\pi)}+\frac{1}{R(0,\theta)}$$ which clearly implies that the curvature is equal to $1$.

%prova: x+y=2=1/x+1/y implica x^2=1. 

The conclusion is:  in the family deformed by the curvature flow starting in a curve of constant width, which is not the circle of radius one,  the dynamic of the billiard map immediately gets more complicated. By Angement's theorem \cite{Ang}, we  get positivity of topological entropy due to the presence of non-Birkhoff periodic orbits and an instability region around the period two orbits.

These special periodic orbits are the subject of the next Section.

\subsection{The family $NP(2n)$ of Normal Periodic Points}

In this section we describe the evolution of a family of periodic orbits close related to the family of period two orbits.

%These orbits are unordered in the sense that they are non Birkhoff periodic orbits, so their presence prevents the existence of invariant rotational (that is homologically non trivial) closed curves in some open region.

\begin{definition}(Normal periodic orbits) We say that $\theta_{0}$ is a normal periodic orbit for a billiard map $B$ if there is an integer $n\geq 2$ and a point  $\theta_{1} \neq \theta_{0}$ such that $B^n(\theta_{0},\frac{\pi}{2})=(\theta_{1},\frac{\pi}{2})$.
\end{definition}

It follows immediately from the definition that $(\theta_{0},\frac{\pi}{2})$ is a periodic point of period $2n$. The trajectory of this periodic orbit starts perpendicular to the curve and after $n-1$ reflections it hits the boundary again orthogonally. Thus, after the reflection it returns tracing back the same polygonal arriving at the initial point.

$NP(2)=B(\Gamma)\cap \Gamma \neq \emptyset$, but as the case  of ellipses with small excentricity, the set $NP(2n)$ may be empty for some $n$. 

In this Section we show  that as the family curves evolves as $t \rightarrow \infty$ by the curvature flow, these orbits gradually disappear. Is is done by describing gemetric necessary conditions for the existence of normal periodic points.

Let $\Gamma$ be the orthogonal wavefront, that is, the subset of initial conditions $(\theta,\frac{\pi}{2})$. $\Gamma$ is the fixed point set of the involution $\tau(\theta,\phi)=(\theta,\pi-\phi)$. Since $B^{-1}=\tau \circ B \circ \tau$,  if $z \in \Gamma$  and $n$ is the smallest positive integer such that  $B^{2n}(z)=z$, then $B^{n}(z)=B^{-n}(z)=\tau \circ B^{n}(z)$, that is, $B^{n}(z) \in \Gamma$.

%The subset $NP(2n)$  of normal periodic points of {\it minimal period} $2n$ is characterized by  $n$ is the smallest positive integer satisfying $B^n(\Gamma) \cap \Gamma$.

Remark:  if there exists a $B^2$ invariant curve $\zeta$, in a disc $V$ around an elliptic diameter $\sigma_{0}$, which is homotopically non-trivial in $V-\{\sigma_{0} \}$, with rational rotation number $\frac{m}{2n}$, then   $\zeta \cap \Gamma \subset NP(2n)$. In other words, if $NP(2n)= \emptyset$ then there is no such invariant curve.

%Symmetries:
%Observe that,  $B^n(\theta_{0},\frac{\pi}{2})=(\theta_{1},\frac{\pi}{2}) $ implies  $B^{n-j}(\theta_{0},\frac{\pi}{2})=R\circ B^{j}(\theta_{n},\frac{\pi}{2}) $. 
%In particular, if $n=2m$ then $B^{n-m}(\theta_{0},\frac{\pi}{2})=R\circ B^{m}(\theta_{n},\frac{\pi}{2}) $. 
%So $A_{m}$ is not injective.

%for the case $n=2m+1$ COMPLETAR

%ONDE COLOCA?

\begin{proposition}
Suppose that for every positive integer $n$, $B^{n}(\Gamma)$ is a graph over $\Gamma$. Then $B(\Gamma)=\Gamma$ and the boundary of the billiard table is a curve of constant width.

\end{proposition}

\begin{proof}
The hypothesis gives $B^{n}(\Gamma) \cap \Gamma=NP(2)$. Let us asumes, by contradiction, that $B(\Gamma) \neq \Gamma$. Let us consider the curve
$B^{2}(\Gamma)$, which by hypothesis is  a graph over $\Gamma$.
By denoting $I=[\sigma_{0},\sigma_{1}] \subset \Gamma$ the maximal interval between two distinct period two points such that $B^{2}(I) \cap I=\{ \sigma_{0}, \sigma_{1} \}$, we have that  the curve $B^{2}(I) \cup R\circ B^{2}(I) $ bounds a disc which is invariant under $B^{2}$ and contains $I$. Since $B^{2}$ is area-preserving, there is an intersection between $B^{2}(I)$ and the interior of $I$ which is  a point in $NP(2n)$, with $n>1$. This is contradiction.
\end{proof}

We now describe the geometric conditions for the existence of normal periodic points in terms of envelope of families of rays.

First we give some estimates:
%The next step is to prove an estimate for the envelope of the wave front $\Vert X(\theta,t) -E_{j}(\theta,t) \Vert$.
\begin{remark} The evolution of  {\it Chords}:

Let $v_{\phi}=\cos(\phi)T(\theta_{0})+\sin(\phi)N(\theta_{0})$ ,$ 0 \leq \phi \leq \frac{\pi}{2}$ and  $X(\theta_{1})=h'(\theta_{1})T(\theta_{1})-h(\theta_{1})N(\theta_{1})$ with $ [X(\theta_{1})-X(\theta_{0})] \wedge v_{\phi} =0$, then

$\langle X(\theta_{1})-X(\theta_{0}),v_{\phi} \rangle=h'(\theta_{1}) \cos(\theta_{1}-\theta_{0}-\phi)+h(\theta_{1})\sin(\theta_{1}-\theta_{0}-\phi)-\cos(\phi)h'(\theta_{0})+\sin(\phi)
h(\theta_{0})$.

 Integration by parts gives:

$\int_{\theta_{0}}^{\theta_{1}} \cos(\theta-\theta_{0} -\phi) R(\theta)d\theta=\int_{\theta_{0}}^{\theta_{1}} \cos(\theta-\theta_{0} -\phi)[h''(\theta)+h(\theta)]d\theta=$

$=\langle X(\theta_{1})-X(\theta_{0}),v_{\phi} \rangle= \Vert X(\theta_{1})-X(\theta_{0}) \Vert$ the length of the chord.

%For example, if $\phi=\frac{\pi}{2}$, then $v_{\phi}$ is the orthogonal wave front and $\langle  X(\theta_{1})-X(\theta_{0}), N(\theta_{0}) \rangle=\int_{\theta_{0}}^{\theta_{1}} \sin(\theta-\theta_{0}) R(\theta)d\theta=h'(\theta_{1}) \sin(\theta_{1}-\theta_{0})+h(\theta_{1}) \cos(\theta_{1}-\theta_{0})+h(\theta_{0})$, restricted to 
%$h'(\theta_{1}) \cos(\theta_{1}-\theta_{0})+h(\theta_{1}) \sin(\theta_{1}-\theta_{0})+h'(\theta_{0})=0$.

 Observe that for $\theta_{1}=\theta_{0}+ \pi$, $\phi=\frac{\pi}{2}$ we obtain the {\it width} $\omega(\theta_{0})=h(\theta_{0}+\pi)+h(\theta_{0})$. Here we generalize  formula proved in the  Lemma 3.9 of \cite{chineses} for the width.

By following  that proof, which we reproduce here, for the sake of completeness, we obtain an estimate for the chord in terms of the {\it entropy} of the curve, ${\cal E}=\frac{1}{2\pi}\int_{0}^{2\pi} \log k(\theta) d\theta$, thus  obtaining the following estimate: 

Claim: There is a constant $C$,  such that $L(\theta_{0},\phi) \geq C e^{{-\cal E}(X)}$.

\begin{proof}

Indeed, by Jensen's inequality, 
$\log L(\theta_{0},\phi) =\log \int_{\theta_{0}}^{\theta_{1}} \vert \cos(\theta-\theta_{0} -\phi) \vert  R(\theta)d\theta $

$\geq \frac{1}{\theta_{1}-\theta_{0}}\int_{\theta_{0}}^{\theta_{1}} \log \vert \cos(\theta-\theta_{0} -\phi)\vert  d\theta -\frac{1}{\theta_{1}-\theta_{0}}\int_{\theta_{0}}^{\theta_{1}}\log[k(\theta)]d\theta-\frac{1}{\theta_{1}-\theta_{0}} \log(\theta_{1}-\theta_{0})$.

Analogously, by integrating from $\theta_{1}$ to $\theta_{0}+ 2\pi$ we obtain:

$\Vert X(\theta_{1})-X(\theta_{0}) \Vert =\int_{\theta_{1}}^{\theta_{0}+ 2\pi} \cos(\theta-\theta_{1} -\phi_{1}) R(\theta)d\theta$,
for $v_{\phi_{1}}=v_{\phi+\pi}$.

$\log L(\theta_{1},\phi+\pi) =\log \int_{\theta_{1}}^{\theta_{0}+2\pi} \vert \cos(\theta-\theta_{1} +\phi+\pi) \vert  R(\theta)d\theta $

$\geq \frac{1}{\theta_{0}+2\pi-\theta_{1}}\int_{\theta_{1}}^{\theta_{0}+2\pi} \log \vert \cos(\theta-\theta_{1} +\phi)\vert  d\theta -\frac{1}{\theta_{0}+2\pi-\theta_{1}}$$\int_{\theta_{1}}^{\theta_{0}+2\pi}\log[k(\theta)]d\theta-$

$\frac{1}{\theta_{0}+2\pi-\theta_{1}} \log(\theta_{0}+2\pi-\theta_{1})$.

 By adding the two estimates we obtain:

$2\log[\Vert X(\theta_{1})-X(\theta_{0}) \Vert \geq C_{0}- \frac{2}{\pi}{\cal E}(X) $, $C_{0}$ independent of the curve.
\end{proof}

The uniform boundedness from below of the chord follows from the fact (Gage-Hamilton \cite{GH} or Lemma 3.8  of \cite{chineses}) that the entropy of the curve $X(.,t)$ {\it decreases monotonically} along the normalized curvature flow, that is  $\frac{d}{dt} {\cal E}(X(.,t) <0$ and the lower boundedness of the curvature for the normalized flow.

\end{remark}

For $n=1$, we consider the intersection between the normal ray and the curve. This intersection is the point $\theta_1$  that satisfies :  $\langle X(t,\theta_{1})-X(t,\theta), T(\theta)\rangle=0$.

By convexity, $\frac{\partial}{\partial \theta_{1}} \langle X(\theta_{1})-X(\theta), T(\theta)\rangle=R(\theta_{1})\langle T(\theta_{1}),T(\theta)\rangle \neq 0$. 
So this equation defines implicitly  a $C^{r}$ function $A_{1}: {\mathbb S}^1 \rightarrow  {\mathbb S}^1$ by $\theta_{1}=A_{1}(\theta)=\pi_{1}\circ B$. 

\begin{proposition}
 $A_{1}$ is a diffeomorphism if and only if the evolute of  $X$ is contained in the region bounded by $X$.
\end{proposition}

\begin{proof}

Indeed, $A_{1}'(\theta) = - \frac{\ell(t, \theta)-R(t, \theta)}{R(A_1)\langle T(A_1),T(\theta)\rangle}$. Hence $A_{1}'(\theta)=0$ if and only if   $\ell(t, \theta)-R(t, \theta) =0$. 

The statement follows from the fact that  $X(\theta)+ R(\theta)N(\theta)$ is the equation of the evolute of $X$.

\end{proof}
Observe that $Fix(A_{1}) = \emptyset$ and the period two orbits of $A_{1}$ correspond exactly to $NP(2)$, period two orbits of the billiard map.

\begin{proposition}
$A_{1}(t, \theta)$  converges uniformly as $t \rightarrow \infty$ to $\pi +\theta$.(antipodal function)
\end{proposition}

\begin{proof}
This  follows from the uniform converge of curvature $k$ and the support function $h$ to 1. 

%Actually
%we these facts imply $\frac{\partial A_1}{\partial t}\rightarrow 0$ and also that the length $S_1(t,\theta)=\langle X(t,A_1)-X(t,\theta),N(\theta)\rangle $ satisfies $\frac{\partial S_1}{\partial t}\rightarrow 0$, 
Recall that, denoting $x_{0}=R(\theta)$, $x_{1}=R(A_{1}(t,\theta)) \sin(\phi_{1}(t,\theta))$, $\phi_{1}$ is the reflected angle with respect to the tangent $T(A_{1}(t,\theta))$, so

 \[
\left(
\begin{array}{c}
A'_{1}(\theta)\\
\phi_{1}'(\theta)
\end{array}
\right)=\frac{1}{x_{1}}\left( 
\begin{array}{c}
x_{0}-l_{0}  \\ 
l_{0}-x_{0}-x_{1}   %
\end{array}%
\right)
\]

$l_{0}=\Vert X(A_{1}(t,\theta))-X(t,\theta) \Vert=\int_{\theta}^{A_{1}(t,\theta)} \sin(\xi -\theta)R(t,\xi) d\xi=$

$-\cos[A_{1}(t,\theta)-\theta]R(A_{1}(t,\theta))+R(t,\theta)+\int_{\theta}^{A_{1}((t,\theta)} \cos(\xi -\theta)R'(t,\xi) d\xi.$

Since $\cos[A_{1}(t,\theta)-\sin(\theta)]=\sin(\phi_{1}(t,\theta))$, we obtain 

$\frac{x_{0}-l_{0}}{x_{1}}=1-\frac{1}{x_{1}} \int_{\theta}^{A_{1}(t,\theta)} \cos(\xi -\theta)R'(t,\xi) d\xi.$

Using the uniform exponential decay of $R'$ we obtain the exponential decay of $\vert  A'_{1}(t,\theta)-1 \vert$, as $t \rightarrow \infty$, since $\frac{1}{x_{1}}$ is bounded.

Moreover, $\phi_{1}'(t,\theta)=1-A'_{1}(t,\theta)$ implies the exponential decay to 0 of this term.
%$$\displaystyle
%\frac{\partial A_1}{\partial t}=\frac{\langle X_t(t,A_1(t,\theta))- X_t(t,\theta),T(\theta)\rangle}{\langle X'(t,A_1(t,\theta)),T(\theta))\rangle}$$

%The numerator is
%$$\left( k-h\right)|_{(t,A_1)}\langle N({A_1}),T(\theta)\rangle -\left( k'-h'\right)|_{(t,A_1)}\langle T({A_1}),T(\theta)\rangle +\left( k'-h'\right)|_{(t,\theta)}$$

%while the denominator is bounded, hence $\lim_{t\to \infty}\frac{\partial A_1}{\partial t}= 0$.

%The second limit follows from
%$$\begin{array}{ll}
%\frac{\partial S_1}{\partial t}(t,\theta)&=\frac{\partial}{\partial t}\langle X(t,A_1 )-X(t,\theta),N(\theta)\rangle\\\\ & =\frac{\partial A_1}{\partial t}\langle X'(t,A_1 ),N(\theta)\rangle+ \langle X_t(t,A_1 )-X_t(t,\theta),N(\theta)\rangle\\\\
%\end{array}
%$$

\end{proof}

%As a Corollary of  Proposition 3.6 we obtain:

\begin{corollary}
For every  $t$ sufficiently large $NP(4)=\emptyset$.
\end{corollary}

\begin{proof} The existence of a point in $NP(4)$, which is not in a period-two orbit, implies the existence of a critical pointof $A_{1}$. 
Indeed,   if $B^{2}(\theta_{0},\frac{\pi}{2})=(\theta_{2},\frac{\pi}{2})$, for $\theta_{2} \neq \theta_{0})$, then
 $B(\theta_{0},\frac{\pi}{2})=\tau \circ B(\theta_{2},\frac{\pi}{2})$. That is: $A_{1}(\theta_{0})=A_{1}(\theta_{2})$ which implies the existence of a critical point $\zeta$, $A'_{1}(\zeta)=0$.

The length of a regular arc of the evolute is the difference of the radius of curvature, 
$\ell_{E}(\theta_{a},\theta_{b})=\vert R(\theta_{a})-R(\theta_{b}) \vert \leq \vert  R'(\zeta) \vert \vert \theta_{a}-\theta_{b} \vert$. Since the point of the evolute corresponding to the vertice of maximal curvature is always inside the region bounded by the curve $X$, the uniform exponential decay of $R'$ imply that for $t$ large enough, the evolute of the curve $X(t,\theta)$ is entirely contained in the region bounded by $X$.

Thus, concluding that the orthogonal wave front focalizes in the interior of the plane region $U(t)$.  In this case $A_{1}$ is a diffeomorphism and $NP(4)=\emptyset$.

\end{proof}

We now proceed for higher periods,$n>2$ to show that for $t>0$ large enough $NP(2n)=\emptyset$.

Let $  B^{m}(\theta, \frac{\pi}{2})=(A_{m}(\theta),\phi_{m}(\theta))$ and denote $x_{m}=R(A_{m}(\theta)) \sin(\phi_{m}(\theta))$, $l_{m}=\Vert X(A_{m+1})-X(A_{m})\Vert$.

\[
DB(A_{m}(\theta),\phi_{m}(\theta))=\frac{1}{x_{m+1}}\left( 
\begin{array}{cc}
l_{m}-x_{m} & l_{m} \\ 
l_{m}-x_{m}-x_{m+1} & l_{m}-x_{m+1}  %
\end{array}%
\right) 
\]

By induction,

\[
\left(
\begin{array}{c}
A'_{m+1}(\theta)\\
\phi_{m+1}'(\theta)
\end{array}
\right)=\frac{1}{x_{m+1}}\left( 
\begin{array}{cc}
l_{m}-x_{m} & l_{m} \\ 
l_{m}-x_{m}-x_{m+1} & l_{m}-x_{m+1}  %
\end{array}%
\right) \left(
\begin{array}{c}
A'_{m}(\theta)\\
\phi_{m}'(\theta)
\end{array}
\right)
\]

Therefore, $A'_{m+1}(\theta)=0$, for $m\geq 1$ if and only if $l_{m}=\frac{x_{m}}{A'_{m}(\theta)+\phi_{m}'(\theta)} A'_{m}(\theta)$.

As in the case $n=1$ we stablish the relatioshionship between the critical points of $A_{m+1}$ and envelopes.

 The reflected angle at $X(t,A_{m}(t,\theta))$, denoted  by $\phi_m(t,\theta)$, defines a wave front
$$\upsilon_{\phi_m(t,\theta)}= \cos (\phi_m(t,\theta)) T(A_{m}(t,\theta))+\sin (\phi_m(t,\theta))N(A_{m}(t,\theta))$$

This  wave front  focalizes in the envelope  of the family of lines $X(A_{m}(t,\theta))+ \lambda v_{\phi_{m}(\theta)}$, which is  defined by $E'_{m}(\theta) \wedge \upsilon_{\phi_m(t,\theta)}=0$.
Equivalently, for $x_{m}=R(t,A_{m}(\theta))\sin(\phi_{m}(t,\theta)) $,

$$E_{m}(t,\theta) =X(A_{m}(t,\theta))+ \frac {x_{m}}{ A'_{m}(t,\theta)+\phi'_{m}(t,\theta)} A'_{m}(t,\theta) \upsilon_{\phi_{m}(t,\theta)}.$$

We conclude that if $A'_{m}(t,\theta)=0$ with $A'_{j}(t,\theta)=0 for j=1,...,m-1$   if and only if the $E_{m}(t,\theta)=X(A_{m+1}(t,\theta))$ an  intersection point between the envelope and the curve.

Therefore the inexistence of normal periodic point is connected with the fact the the envelope is entirely contained in the region bounded by the curve $X$.

\begin{proposition}
For $t$ sufficiently large, the envelope $E_{m}$ is entirely contained in the interior of the region $U(t)$ bounded by the curve $X(t,.)$.

\end{proposition}

\begin{proof}

If we let $\beta_{m}(t,\theta)= \frac{\phi_{m}'(t,\theta)}{A'_{m}(t,\theta)}$, then $$\beta_{m}(t,\theta)=1-x_{m} \frac{1-\beta_{m-1}(t,\theta)}{l_{m-1}-x_{m-1}-l_{m-1}\beta_{m-1}(t,\theta)}.$$

Since $lim_{t \rightarrow \infty} l_{m}=2$ and   $lim_{t \rightarrow \infty} x_{m}=1$ we obtain, by induction, $lim_{t\rightarrow \infty}\beta_{m}(t,\theta)=0$.

The conclusion follows from the expression 
$$E_{m}(t,\theta) =X(A_{m}(t,\theta))+ \frac {x_{m}}{ 1+ \beta_{m}(t,\theta)} \upsilon_{\phi_{m}(t,\theta)}.$$

\end{proof}

Let us assume that $NP(2n) \neq \emptyset$ we are going to show that there is $j$ such that $A_{j}$ is not a diffeomorphism.

There are two cases: $n=2k$ and $n=2k+1$.

For the case $n=2k$, we have  $B^{4k}(x)=x$ or $\tau \circ B^{2k} \circ B^{2k} (x)=x$, assume that $k$ is the smallest such integer.

Since $\tau \circ B \circ \tau = B^{-1}$ and $y=B^{2k}(x) \neq x$,  we have $\tau \circ B^{k}(B^{k}(x))=y$ or  $\tau (B^{k}(x))=B^{k}(y)$, that is, $A_{k}=p_{1} \circ B^{k}$ is not injective  so  there is a point $z \in\Gamma $ such that  $A'_{k}(z)=0$.

In the second case, $B^{4k+2}(x)=x$ or $B^{4k}(B^{2}(x))=x  \in \Gamma= Fix(\tau)$.

Then  $\tau \circ B^{2k} \circ B^{2k} (B^{2}(x))=x$ implies $\tau  \circ B^{2k}(B^{2}(x)=B^{2k}(x)$. Since $B^{2k}(B^{2}x)) \neq B^{2k}(x)$, this implies , the existence of a point $z$ such that $A'_{2k}(z)=0$.

\begin{theorem}
Given a positive integer $j$ there exists a positive number $t_j$ such that for all $t >t_j$  the function $\theta \mapsto A_{m}(t,\theta) $ is  a diffeomorphism for $m=1,2,\cdots, j$.
\end{theorem}
%\begin{proof} In order to simplify the notation we let ${\upsilon}_{\phi_j}=\upsilon_j $, $X_0=X(t,\theta)$, $X_j=X(t,A_j(\theta))$.

\begin{corollary}
For all $t>t_j$, obtained in the previous theorem  there $NP(2n) =\emptyset$, $2 \leq n \leq 2j$ for the billiard maps associated to the curves $X(t,\theta)$. In particular, there is no invariant curve of period $2n$ surrounding an elliptic diameter.
\end{corollary}

{\it The curvature flow destroys the resonant invariant curve of type $2n$ surrounding an elliptic period-two orbit.}
Observe that an invariant curve of period $2n$ surrounding an elliptic period-two orbit must intersect $\Gamma$ in  $NP(2n)$. Therefore, there is no resonant curve of period $2n$ around an eliptical diameter.

 The  above result applies to the case of an ellipse  with small excentricity. All the above envelopes as strictly contained in the convex region $U$, so $NP(2n)=\emptyset$.

\begin{proposition}
 Elipses with small eccentricity, the normal periodic orbits which are not diameters have large periods.
\end{proposition}

%Actually, using the above calculations we have the following

%\begin{proposition}
%If define a sequence $(a_{n})$ of positive numbers  such that for the ellipse we have if  $NP(2n) \neq \emptyset$, for $n>2$, then $a \geq (a_{n})$.
%\end{proposition}

\begin{proposition}
Let $\Gamma=\{(\theta,\frac{\pi}{2}\}$, the set of initial conditions perpendicular to the boundary curve $\gamma$.
If $B^{n}(\Gamma)$ is a graph for every $n>0$ then $X$ has constant width.
\end{proposition}

Since there is a time $t_j>0$ such that for all $t>t_j$ the map $A_n$ is a circle diffeomorphism with fixed points (the diameters), the exist not a periodic orbit of higher period. Thus  $NP(2n)=\emptyset$ . 

Observe that an invariant curve of period $2n$ surrounding an elliptic period-two orbit must intersect $\Gamma$. Therefore it must contain a point of $NP(2n)$. A contradiction.
%\end{proof}

{\it The curvature flow destroys the resonant invariant curve of type $2n$ surrounding an elliptic period-two orbit.}

 The  above result applies to the case of an ellipse  with small excentricity ($a$ close to 1). All the above envelopes as strictly contained in the convex region $U$, so $NP(2n)=\emptyset$.

\begin{proposition}
 Elipses with small eccentricity, the normal periodic orbits which are not diameters have large periods.
\end{proposition}

%Actually, using the above calculations we have the following

%\begin{proposition}
%If define a sequence $(a_{n})$ of positive numbers  such that for the ellipse we have if  $NP(2n) \neq \emptyset$, for $n>2$, then $a \geq (a_{n})$.
%\end{proposition}

\begin{proposition}
Let $\Gamma=\{(\theta,\frac{\pi}{2}\}$, the set of initial conditions perpendicular to the boundary curve $\gamma$.
If $B^{n}(\Gamma)$ is a graph for every $n>0$ then $X$ has constant width.
\end{proposition}

\subsection{ Destruction of non-convex caustic of the of period four of ellipses}

In this Section, we prove that the perturbation of an ellipse by the curvature flow destroys the invariant curve of period four around the smallest diameter(period two elliptic orbit).
The argument used is elementary and direct. We show that the normal periodic orbit of period four becomes hyperbolic.  

As a consequence, there is no invariant curve of period four around the elliptic period two orbit. So, it follows, by a Therorem of S. Agenent \cite{Ang}, that {\it the topological entropy of the billiard map is positive}.
So the deformation by the curvature flow of the Billiard map of the ellipse produces a Birkhoff instability, with a non-trivial hyperbolic set around the elliptic diameter.

In the next Section, using a general argument, similar to \cite{D-DC-RR}, we show that, the curvature flor destroys every resonant invariant curve around the elliptic diameter, thus, producing many instability regions.

%Question: is Kaloshi-Sorrentino rigidity theorem valid ? Namely, one considers the dynamics around the eliptic period-two orbit?

\begin{example}
 An elementary computation shows that the invariant curves corresponding to period four orbits around the elliptic diameter do not persist under the curvature flow.
Any homotopically trivial invariant curve around the diameter $\sigma_{0}$  contains a normal periodic point of period four $NP(4)$. 

We show that under the curvature flow pertubation, any $NP(4)$ is {\it hyperbolic}
\end{example}

%Remark:

The above argument can be replaced by using the generating function of $B^{2}$. This is what is done in the next Section.For that, we first prove the following Lemma  on the "conjunction" of two generating functions see(\cite{mather1}):

\begin{lemma}
Let  $h_{1}$ and $h_{2}$ be generating functions of two exacts area preserving twist maps $F_{1}$ and $F_{2}$ respectively.
Suppose that,$\frac{{\partial}^2}{\partial x_{1}^2}h_{1}(x_{0},x_{1})+ \frac{{\partial}^2}{\partial x_{1}^2}h_{2}(x_{1},x_{2})\neq 0$.
Then the composition $F_{2} \circ F_{1}$ is generated by $h(x_{0},x_{2})= h_{1}(x_{0},x_{1})+h_{2}(x_{1},x_{2})$, restricted to the set
$\frac{\partial}{\partial x_{1}}h_{1}(x_{0},x_{1})+ \frac{\partial}{\partial x_{1}}h_{2}(x_{1},x_{2})=0$.
\end{lemma}

\begin{proof}

By the Implicit Function Theorem, $\frac{\partial}{\partial x_{1}}h_{1}(x_{0},x_{1})+ \frac{\partial}{\partial x_{1}}h_{2}(x_{1},x_{2})=0$ defines a surface $x_{1}=g(x_{0},x_{2})$ if it is non-empty and $\frac{{\partial}^2}{\partial x_{1}^2}h_{2}(x_{0},x_{1})+\frac{{\partial}^2}{\partial x_1{}^2}h_{2}(x_{1},x_{2}) \neq 0$.

If this is the case, then 
$-\frac{\partial}{\partial x_{0}}h(x_{0},x_{2})=-[\frac{\partial}{\partial x_{0}}h_{1}(x_{0},x_{1})+\frac{\partial}{\partial x_{1}}h_{1}(x_{0},x_{1})\frac{\partial x_{1}}{x_{0}}
+\frac{\partial}{\partial x_{1}}h_{2}(x_{1},x_{2})\frac{\partial x_{1}}{x_{0}}]$
$=-\frac{\partial}{\partial x_{0}}h_{1}(x_{0},x_{1})=-y_{0}$, using  the restriction condition.

This condition also gives

$\frac{\partial}{\partial x_{2}}h(x_{0},x_{2})=\frac{\partial}{\partial x_{1}}h_{1}(x_{0},x_{1})\frac{\partial x_{1}}{x_{2}}+\frac{\partial}{\partial x_{1}}h_{2}(x_{1},x_{2})\frac{\partial x_{1}}{x_{2}}+\frac{\partial}{\partial x_{2}}h_{2}(x_{1},x_{2})=y_{2}$

\end{proof}

{\bf Remark}: For the billiard map the generating function is $L(s,s')=\Vert X(s')-X(s) \Vert$ we have 
$\frac{{\partial}^2 L}{\partial {s'}^2}(s,s')+\frac{{\partial}^2 L}{\partial {s''}^2}(s',s'')=$
$2{\sin}(\phi)[{\sin}(\phi)(\frac{1}{L(s,s')}+\frac{1}{L(s',s'')})+k(s')]>0$,
where $\cos(\phi)=\langle\frac{X(s')-X(s)}{L(s,s')},X'(s')\rangle=\langle\frac{X(s'')-X(s')}{L(s',s'')},X'(s')\rangle$

So the generating function of $B^{2}$ is 

$h(s,s'')=L(s,s')+L(s',s'')$ restrited to $\frac{\partial }{\partial s'}[L(s,s')+L(s',s'')]=0$.

\subsection{Destruction of invariant curves around the elliptic diameter}

In this Section we show that one may adapt  the arguments  used in \cite{D-DC-RR}  to show the following:

\begin{theorem}
The deformation of an ellipse along the curvature flow destroys all periodic  invariant curves of the associated billiard map around the elliptic diameter (Non-convex Ressonant Caustics).
\end{theorem}

\begin{proof}
We start by considering the Melnikov potential applied to the generating function correponding to $B^2$. 

We are now interested in trajectories tangent to {\it Non-convex  Caustics} 

$$\frac{x^2}{a^2-{\lambda}^2}+\frac{y^2}{b^2-{\lambda}^2}=1$$ with $b<\lambda <a$.

Write, as above,  the deformed ellipse by the curvature flow as $X(\epsilon,\varphi)=(\cosh(\mu_{\epsilon}(\varphi) \cos(\varphi),\sinh(\mu_{\epsilon}(\varphi) \sin(\varphi))$ with
$\mu_{\epsilon}(\varphi)=\mu_{0}+\mu_{1}(\varphi)+ O({\epsilon}^2)$, $\cosh(\mu_{0})=a$,$\sinh(\mu_{0})=b$ and 

$$\mu_{1}(\varphi)=\frac{-ab}{{(a^{2} {\cos}^{2}(\varphi)+ {\sin}^{2}(\varphi))}^{2}}$$

Let us consider the {\it Action} 

$$W_{\epsilon}^{p,2q}(\varphi_{0},\varphi_{2},\cdots,\varphi_{2q})= \sum_{j=0}^{2q-1} L_{\epsilon}(\varphi_{j},\varphi_{j+1})$$

restricted to

$\frac{\partial }{\partial \varphi_{2j+1}}[L_{\epsilon}(\varphi_{2j+1},\varphi_{2j})+L_{\epsilon}(\varphi_{2j+2},\varphi_{2j+1})]=0$, for $j=0,\cdots,n-1$.

By writing $L_{\epsilon}=L_{0}+\epsilon L_{1}+ O({\epsilon}^2)$, and  the generating function for $B_{\epsilon}^{2}$ as $h_{\epsilon}=h_{0}+\epsilon h_{1}+ O({\epsilon}^2)$ we calculate
$h_{1}(\varphi,\varphi'')=L_{1}(\varphi,\varphi')+L_{1}(\varphi',\varphi'')$ restricted to $\frac{\partial }{\partial \varphi'}[L_{0}(\varphi,\varphi ')+L_{0}(\varphi',\varphi'')]=0$.

 This last equation defines  a smooth function $\varphi'=\Psi(\varphi, \varphi'')$, hence,we write: 

$$h_{1}(\varphi,\varphi'')=L_{1}(\varphi,\Psi(\varphi, \varphi''))+L_{1}(\Psi(\varphi, \varphi''),\varphi'')$$
To obtain the $(p,2q)$- subhamonic Melnikov potential we write:

Recall that $L_{1}(\varphi,\varphi')=\langle \frac{X_{0}\varphi')-X_{0}(\varphi)}{L(\varphi,\varphi')},X_{1}(\varphi')-X_{1}(\varphi)\rangle $
where the deformed curve is $X(\epsilon,.)=X_{0}+\epsilon X_{1}+O({\epsilon}^2)$, so $X_{1}(\varphi)=a\mu_{1}(\varphi)D^{-2}X_{0}(\varphi)$.

$$W_{\epsilon}^{p,2}(\varphi_{0},\varphi_{2},\cdots,\varphi_{2q})=W_{0}^{p,2q} +\epsilon W_{1}^{p,2q} + O({\epsilon}^2) $$ 
with $$W_{1}^{p,2q}=\sum_{j=0}^{2q-1} h_{1}(\varphi_{2j},\varphi_{2j+2})=\sum_{j=0}^{2q-1} L_{1}(\varphi_{j},\varphi_{j+1})$$
restrited to $\frac{\partial }{\partial \varphi_{2j+1}}[L_{0}(\varphi_{2j+1},\varphi_{2j})+L_{0}(\varphi_{2j+2},\varphi_{2j+1})]=0$, for $j=0,\cdots,q-1$.

We are ready to apply the  results of Ramires-Ros et. al : 

{\bf Claim 1} Lemma 10 of \cite{SP}

Let $p_{j}=\frac{X_{0}(\varphi')-X_{0}(\varphi)}{L_{0}(\varphi,\varphi')}$ and $C_{\lambda}$ be the Caustic (hyperbole) tangent to the billiard trajectory inside the elipse determined by $(X_{0}(\phi_{j}))$. Then

$$a\langle p_{j-1}-p_{j}, D^{-2}X_{0}(\phi_{j})\rangle=2\lambda$$
where $D=diag(a,1)$, the diagonal matrix.

{\bf Claim 2} The $(p,2q)$-subharmonic Melnikov  potential is: $$\sum_{j=0}^{2q-1} L_{1}(\varphi_{j},\varphi_{j+1})=2\lambda\sum_{j=0}^{2q-1}\mu_{1}(\phi_{j}) $$
restricted to $\frac{\partial }{\partial \varphi_{2j+1}}[L_{0}(\varphi_{2j+1},\varphi_{2j})+L_{0}(\varphi_{2j+2},\varphi_{2j+1})]=0$, for $j=0,\cdots,q-1$.

{\it Proof}

Using the above notation we write $\varphi_{2j+1}=\Psi_{\epsilon}(\varphi_{2j}, \varphi_{2j+2})$,
$$h_{1}(\varphi_{2j},\varphi_{2j+2})=L_{1}(\varphi_{2j},\Psi_{0}(\varphi_{2j}, \varphi_{2j+2}))+L_{1}(\Psi_{0}(\varphi_{2j}, \varphi_{2j+2}),\varphi_{2j+2})$$

But $L_{1}(\varphi_{j},\varphi_{j+1})=\langle p_{j},X_{1}(\varphi_{j+1})-X_{1}(\varphi_{j})\rangle=$

$=ab\langle p_{j},\mu_{1}(\varphi_{j+1})D^{-2}X_{0}(\varphi_{j+1})-\mu_{1}(\varphi_{j})D^{-2}X_{0}(\varphi_{j})\rangle$, so proceeding as in \cite{SP} 

 %$L_{1}(\varphi_{2j},\varphi_{2j+1})+L_{1}(\varphi_{2j+1},\varphi_{2j+2})=$

%$<p_{2j},X_{1}(\varphi_{2j+1})-X_{1}(\varphi_{2j})>+<p_{2j+1},X_{1}(\varphi_{2j+2})-X_{1}(\varphi_{2j+1})>=$

%$-<p_{2j},X_{1}(\varphi_{2j})>+<p_{2j}-p_{2j+1},X_{1}(\varphi_{2j+1})>+<p_{2j+1},X_{1}(\varphi_{2j+2})>=$

%$-<p_{2j},X_{1}(\varphi_{2j})>+2\lambda\mu_{1}(+<p_{2j+1},X_{1}(\varphi_{2j+2})>=$

$$\sum_{j=0}^{q-1}h_{1}(\varphi_{2j},\varphi_{2j+2})=$$
$$\sum_{j=0}^{q-1} L_{1}(\varphi_{2j},\Psi_{0}(\varphi_{2j}, \varphi_{2j+2})+L_{1}(\Psi_{0}(\varphi_{2j}, \varphi_{2j+2}),\varphi_{2j+2})=$$
$$2\lambda[\sum_{j=0}^{q-1}\mu_{1}(\varphi_{2j})+\sum_{j=0}^{q-1}\mu_{1}(\Psi_{0}(\varphi_{2j},\varphi_{2j+2}))]$$

{\bf Claim 3} $ h_{1}$ is  not identically zero.

Proof: We now use \cite{D-DC-RR} and a time reparametrization so that the dynamics along an invariant  RIC of type  $(p,2q)$ becomes a shift $t \mapsto t+\delta(\lambda)$
For this, we calculate the period and modulus of the corresponding elliptic function associated to hyperbolic caustic, see \cite{C-F}

 We also use the fact that, for the ellipse,  $\Psi_{0}$ is an analytic function and we proceed by calculating  the Laurent expansion as in Lemma 5.1 of \cite{D-DC-RR}.

Here are  some details:

Let $k^{2}(\lambda)=\frac{a^{2}- \lambda^{2}}{a^{2}-b^{2}}$, with $b<\lambda<a$, the defining parameter of the hyperbolic caustic. The caustic corresponds to a periodic orbit of type 
$\frac{p}{2q} \in \mathbb{Q}$ if and only if $\sqrt{1-\frac{1}{\lambda}}=cn(\frac{p}{q}K(\lambda))$, with $K(k)=\int_{0}^{\frac{\pi}{2}} \frac{du}{\sqrt{1-k^{2}{\sin}^{2}(u)}}$,
$\sin(\frac{\zeta}{2})=\lambda$.

  $k=k(\lambda)$ is the {\it modulus} and $\delta(\lambda)=2\int_{0}^{\frac{\zeta}{2}}\frac{du}{\sqrt{1-k^{2}{\sin}^{2}(u)}}$  is the  associated {\it period}.

 $t=F(\varphi,k)= \int_{0}^{\varphi} \frac{1}{\sqrt{1-k^{2} {\sin}^{2} u}} du$ defines a time reparametrization such that the dynamics along an orbit associated to $C_{\lambda}$ becomes a shift
$t \mapsto t+\delta(\lambda)$.

In Lemma 5.1 of \cite{D-DC-RR} it is been proven that the  function
$$\tilde{\mu}_{1}(t)=\frac{-ab}{{(a^{2} {cn}^{2}(t)+ b^2{sn}^{2}(t))}^{2}}$$

can be extended to ${\mathbb C}$ as an elliptic function of order four with  poles of order two at

$T_{\pm}=[\pm \frac{\zeta}{2} +K'i]+2K{\mathbb Z}+2K'{\mathbb Z},\zeta=2K-\delta$.

And Laurentz expansion
$$\tilde{\mu}_{1}(t_{\pm}+\tau)=\frac{\alpha_{2}}{{\tau}^{2}}\pm \frac{\alpha_{1}}{\tau}+O(1), \tau \rightarrow 0$$

The fact that $\alpha_{2} \neq 0$ implies, as in Proposition 5.3 of \cite{D-DC-RR} that Melnikov potential is not constant.

To prove this we write the sub-harmonic potential in the variable $t$ as

$${\tilde L}_{1}(t,\delta)=2\lambda[\sum_{j=0}^{q-1}\mu_{1}(t+ 2j \delta)+\sum_{j=0}^{q-1}\mu_{1}(\Psi_{0}(t+2j \delta,t+(2j+2)\delta)]$$.

Recall that, in our case, the twist condition implies that the function $\Psi_{0}$ is {\it analytic} with $t \mapsto \Psi_{0}(t+2j \delta,t+(2j+2)\delta)$ a diffeomorphism.

Therefore, according to Proposition 5.3, the  poles of the complex extension of $(t,\delta)$ are defined by:
either $t+2j \delta \in {\cal T}$ or $\Psi_{0}(t+2j \delta,t+(2j+2)\delta) \in {\cal T}$.

% Hence we can adapt the proof of Proposition 5.3 in \cite{D-DC-R}   to show that {\bf the Melnikov potential is not constant} and so the non-convex resonant caustic does not persist under the perturbation $\mu_{\epsilon}(\varphi)$.

It remains to show that the non-vanishing of the potential implies the breaking of the invariant curves. For that, we use the action-angle coordinates $(x,y)$, in a neighborhood of the elliptic diameter,  to write the invariant curve as a graph $(x,v(x))$ and apply \cite{SP}. 

Let us denote the generating function of this change of coordinates $\Psi$ as $g(x_{0},x_{1})$,which is independent of $\epsilon$.

The composition ${\Phi}^{-1} \circ B_{\epsilon}^2 \circ \Phi$ is generated by $G(x,x_{1},\epsilon)=-g(\varphi_{2},x_{1})+L_{\epsilon}(\varphi,\varphi_{2})+g(x,\varphi)$ restricted to
${\partial}_{1} L_{\epsilon}(\varphi,\varphi_{2})+{\partial}_{2} g(x,\varphi)=0$, defining $\varphi_{\epsilon}=\varphi(x,\varphi_{2},\epsilon)$ and 
$-{\partial}_{1} g(\varphi_{2},x_{1})+{\partial}_{2} L_{\epsilon}(\varphi_{\epsilon},\varphi_{2})=0$, which defines $\varphi_{2}=\varphi_{2}(x,x_{1},\epsilon)$.

%${\partial}_{1} L_{\epsilon}(\varphi,\varphi_{2})+{\partial}_{2} g(x,\varphi)=0$, which defines  $\varphi(x,\varphi_{2})$ and $-{\partial}_{1} g(\varphi_{2},x_{1})+\frac{\partial}{\partial \varphi_{2}}[ L_{\epsilon}(\varphi(x,\varphi_{2}),\varphi_{2})+ g(x,\varphi(x,\varphi_{2}))]=0$.

Therefore

 $G(x,x_{1},\epsilon)=-g(\varphi_{2}(x,x_{1},\epsilon),x_{1})+L_{\epsilon}(\varphi(x,\varphi_{2}(x,x_{1},\epsilon),\varphi_{2}(x,x_{1},\epsilon)))+g(x,\varphi(x,\varphi_{2}(x,x_{1},\epsilon))$

Writing $G(x,x_{1},\epsilon)=G_{0}(x,x_{1})+ \epsilon G_{1}(x,x_{1})+ O({\epsilon}^2)$ with $G_{1}(x,x_{1})=\frac{\partial}{\partial \epsilon}|_{\epsilon=0}L_{\epsilon}(\varphi,\varphi_{0})=h_{1}(\varphi.\varphi_{2})$.

 This gives Claim 2 and finishes the proof of Theorem.
\end{proof}

%\end{document}

Josu\'e Damasceno-UFOP-Minas Gerais-Brazil: josue@ufop.edu.br

M\'{a}rio J. Dias Carneiro- UFMG- Minas Gerais - Brazil: carneiro@mat.ufmg.br

Carlos Salazar -CEFET-Minas Gerais-Brazil- carlosleo17@hotmail.com

\end{document}